\documentclass{article}
\usepackage[utf8]{inputenc}
\usepackage{amsmath,mathtools,amsthm,amssymb, xcolor}
\usepackage{cite}
\usepackage{tikz}
\usetikzlibrary {shapes.multipart}
\usetikzlibrary {calc}
\usetikzlibrary{arrows.meta}
\usepackage{amsfonts}
\usepackage{thm-restate}

\usepackage{comment}
\usetikzlibrary{decorations.pathreplacing,calligraphy}
\usetikzlibrary{decorations}
\pgfdeclaredecoration{simple line}{initial}{
	\state{initial}[width=\pgfdecoratedpathlength-1sp]{\pgfmoveto{\pgfpointorigin}}
	\state{final}{\pgflineto{\pgfpointorigin}}
}

\usepackage{fullpage}

\usepackage[pdftex,hypertexnames=false,linktocpage=true]{hyperref}
\hypersetup{colorlinks=true,linkcolor=black,anchorcolor=blue,citecolor=black,filecolor=blue,urlcolor=blue,bookmarksnumbered=true,pdfview=FitB}

\newtheorem{theorem}{Theorem}

\newtheorem{lemma}[theorem]{Lemma}

\theoremstyle{definition}

\newtheorem{remark}[theorem]{Remark}
\newtheorem*{sketch}{Proof sketch (omitting compression of twins)}

\usepackage{graphicx}

\newcommand{\reduce}{\text{red}}

\newcommand{\universal}{universal }

\DeclarePairedDelimiter\set{\{}{\}}
\DeclarePairedDelimiter\abs{|}{|}

\title{Shortened Universal Cycles for Permutations}

\author{Rachel Kirsch\thanks{Department of Mathematical Sciences, George Mason University, Fairfax, VA, E-mail: \texttt{rkirsch4@gmu.edu}. Research of this author was partially supported by NSF grant DMS-1839918.}
	\and
	Bernard Lidick\'{y}\thanks{Department of Mathematics, Iowa State University, Ames, IA, E-mail: \texttt{lidicky@iastate.edu}. Research of this author was partially supported by NSF grants DMS-2152490 and DMS-1855653 and Scott Hanna fellowship.}
	\and 
	Clare Sibley\thanks{Houston, TX, E-mail: \texttt{clare.sibley@rice.edu}. Research of this author was partially supported by NSF grant DMS-1839918.}
	\and
	Elizabeth Sprangel\thanks{Department of Mathematics, Iowa State University, Ames, IA, E-mail:  \texttt{sprangel@iastate.edu}. Research of this author was partially supported by NSF grant DMS-1839918.}}
\date{October 13, 2022}

\begin{document}
	\maketitle
	
	\begin{abstract}
		Kitaev, Potapov, and Vajnovszki [On shortening u-cycles and u-words for permutations, Discrete Appl. Math, 2019] described how to shorten universal words for permutations, to length $n!+n-1-i(n-1)$ for any $i \in [(n-2)!]$, by introducing incomparable elements. They conjectured that it is also possible to use incomparable elements to shorten universal cycles for permutations to length $n!-i(n-1)$ for any $i \in [(n-2)!]$. In this note we prove their conjecture. The proof is constructive, and, on the way, we also show a new method for constructing universal cycles for permutations.
	\end{abstract}
	
	\section{Introduction}
	A \emph{universal cycle} for a family $\mathcal{F}$ of combinatorial objects is a cyclic sequence whose consecutive substrings of a given length $n$ represent each object of $\mathcal{F}$ exactly once. The canonical type of universal cycle is a De Bruijn sequence, which is a universal cycle for the words of length $n$ over an alphabet $\mathcal{A}$. De Bruijn sequences have been widely studied \cite{uniwww}. Chung, Diaconis, and Graham~\cite{CDG92} introduced the notion of universal cycles 
	for other combinatorial objects such as permutations, sets, and set partitions. A \emph{universal word}  
	for $\mathcal{F}$ is the non-cyclic analogue of a universal cycle. A universal cycle has length $\abs{\mathcal{F}}$, while a universal word has length $|\mathcal{F}|+(n-1)$. 
	
	In this paper, we are focusing on permutations. 
	We use $[n]$ to denote $\{1,2,\ldots,n\}$ and $S_n$ to denote all permutations of $[n]$.
	A \emph{universal word} for $S_n$ is a word $w$ over $\mathbb{N}$ such that each permutation in $S_n$ is order-isomorphic to exactly one consecutive substring of length $n$. 
	Notice that entries in $w$ are not restricted to $[n]$. 
	For example, $14524314$ is a universal word for $S_3$. The permutations in $S_3$ represented by $w$ from left to right are $123$, $231$, $312$, $132$, $321$, and $213$. See Figure~\ref{fig:intro} for a way to plot the word that depicts the order of its letters. We use similar figures throughout the paper to indicate relative order of letters.
	A \emph{universal cycle} for $S_n$ is a cyclic universal word. For example, $145243$ is a universal cycle for $S_3$. Note that a universal word has length $n!+(n-1)$, and a universal cycle has length $n!$. Hurlbert \cite{H90} showed that universal cycles for $S_n$ exist for all $n$. 
	Chung, Diaconis, and Graham~\cite{CDG92} conjectured that it is sufficient to use $n+1$ distinct numbers in a universal cycle for $S_n$, which would be best possible for $n \geq 3$.
	This conjecture was proved constructively by Johnson~\cite{J09}. In addition, constructions of universal cycles for permutations, and variations thereof, have been studied in \cite{J93, GKSZ19, RW10, HRW12, W17, HH13weak, HH13perm, AW09, GSWW20, KPV19}.
	Universal cycles for permutations have applications in various areas such as
	molecular biology~\cite{Compeau2011},
	computer vision~\cite{Pags2005},
	robotics~\cite{Scheinerman2001}, and
	psychology~\cite{Sohn1997}.
	
	\tikzset{
		pvtx/.style={inner sep=1.7pt, outer sep=0pt, circle, fill=black,draw=black}, 
	}
	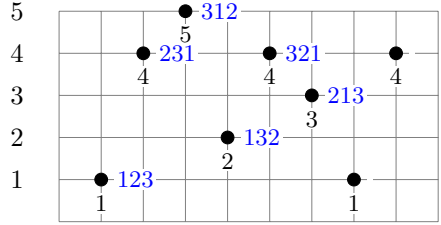
\begin{figure}
		\begin{center}
			\begin{tikzpicture}[scale=0.56]
				\draw[help lines](0,0) grid (9,5);
				\foreach \x/\y/\z in {1/1/123,2/4/231,3/5/312,4/2/132,5/4/321,6/3/213,7/1/ ,8/4/ }{
					\draw(\x,\y) node[pvtx,label=right:{ }]{};
					\draw(\x,\y-0.3) node[below,fill=white,inner sep=1pt]{\small \y};
					\draw(\x+0.3,\y) node[right,fill=white,inner sep=1pt]{\color{blue}\small \z};
				}
				\foreach \y in {1,2,3,4,5}{
					\draw (-1,\y) node{\y};
				}
			\end{tikzpicture}
		\end{center}
		\caption{A universal word 14524314 for $S_3$ depicted in a grid from left to right. A blue label to the right of a node indicates a permutation of $S_3$ starting at the node.}
		\label{fig:intro}
	\end{figure}

	Universal cycles, for permutations and more generally for $\mathcal{F}$, are useful because they represent the elements of $\mathcal{F}$ compactly. In recent years there has been interest in shortening universal cycles to compress information even further. (These efforts are not to be confused with the study of shorthand universal cycles for permutations, which are shorthand in the sense of using only $n$ distinct symbols but have the same length, $n!$, as universal cycles for $S_n$; see \cite{HRW12, RW10, GSWW20}. Here shortening means reducing the length of the cycle.) 
	De Bruijn sequences have been shortened, to \emph{universal partial cycles} and \emph{universal partial words}, using a wildcard symbol $\diamond$ that covers any letter of the alphabet, so that a window of length $n$ may cover more than one word of length $n$; see \cite{CKMS17,G18,FGKMM21}. Graph universal cycles, introduced in \cite{BKS10}, and graph universal cycles for permutations, introduced in \cite{CGGP21}, have been shortened in \cite{KSS22}.
	
	Kitaev, Potapov, and Vajnovszki~\cite{KPV19} shortened universal words for permutations in two different ways: with wildcard symbols and with incomparable elements. Similarly to universal partial words and universal partial cycles, they considered using the wildcard symbol $\diamond$, which yielded many nonexistence results, but they created shortened universal words for permutations using a wildcard symbol $\diamond_{\set{a,b}}$ that covers either of the two elements $a$ and $b$.
	
	Most relevantly to our work, in \cite{KPV19}, Kitaev, Potapov, and Vajnovszki used incomparable elements at distance $n-1$ to shorten universal words for permutations of $[n]$ to lengths $n! +n-1-i(n-1)$ for each $i \in [(n-2)!]$. A word $w$ of length $n$ with incomparable elements \emph{covers} all permutations of length $n$ that are linear extensions of the order given by $w$.
	In the case of incomparable elements at distance $n-1$, $w$ covers two permutations. 
	For example, the word $2132$ covers the two permutations $2143$ and $3142$. We also say that a longer word $v$ containing $w$ as a consecutive substring \emph{covers} the permutations that $w$ covers. For example, $4321324$ contains $2132$ as a consecutive substring, so $4321324$ covers $2143$ and $3142$; the other $4$-permutations that $4321324$ covers are $4321$, $3214$, $4213$, and $1324$.
	
	\begin{theorem}[Kitaev, Potapov, and Vajnovszki~\cite{KPV19}]\label{thm:uword}
		Using incomparable elements at distance $n-1$, one can obtain shortened universal words for $S_n$ of lengths 
		$n!+n-1-i(n-1)$ for each $0 \leq i \leq (n-2)!$.
	\end{theorem}

	They conjecture \cite[Conjecture 8]{KPV19} that their result may be strengthened by obtaining shortened universal cycles instead of shortened universal words. Here we prove their conjecture.
	
	\begin{restatable}{theorem}{MainThm}\label{thm:conj8}
		For $n \ge 3$ and each $0 \leq i \leq (n-2)!$, using incomparable elements at distance $n-1$, one can obtain a shortened universal cycle for $S_n$ of length $n!-i(n-1)$.
	\end{restatable}
	
	Our proof is constructive and does not attempt to control the number of distinct symbols used. Examples of shortened universal cycles for $S_4$ arising from the construction include \[(1, 2, 3, 12, 4, 10, 9, 8, 11, 7, 8, 6, 9, 10, 6, 5, 9, 7, 10, 8, 3) \quad \text{and} \quad(1, 2, 3, 12, 6, 8, 7, 6, 8, 5, 9, 8, 10, 11, 8, 10, 6, 3).\]
	See the Appendix for more detailed illustrations. The running time is $O((n!)^2)$, which comes from finding the Euler tour in a graph with $n!$ edges and $(n-1)!$ vertices in quadratic time, and finding a cyclic word representing the tour, which involves relabeling at most $n!$ symbols at most $n!$ times. The memory required to store the output is $O(n!(\log(n!)))$.
	
	Getting from Theorem~\ref{thm:uword} to Theorem~\ref{thm:conj8} requires several background definitions which we provide forthwith.
	For a word $w$ over comparable letters, we denote by $\reduce(w)$ the word obtained from $w$ by replacing each copy of the $i$th smallest element of $w$ by $i$. For example, $\reduce(37361) = 24231$. We say that $w$ \emph{reduces to} $\reduce(w)$, and words $w$ and $v$ satisfying $\reduce(w) = \reduce(v)$ are \emph{order-isomorphic}.
	
	For any $n$, the \emph{graph of overlapping $n$-permutations} 
	is a directed graph on $n!$ vertices, each vertex corresponding to one permutation in $S_n$.
	There is an edge from $x=x_1x_2\cdots x_n$ to $y = y_1y_2\cdots y_n$ iff $\reduce(x_2x_3\cdots x_{n}) = \reduce(y_1y_2\cdots y_{n-1})$,
	see Figure~\ref{fig:overlapgraph}.
	Notice that a \universal{}cycle for $S_n$ gives a Hamiltonian cycle in the graph of overlapping $n$-permutations. 
	
	\begin{figure}
		\begin{center}
			\tikzset{
				similarp/.style={fill=gray!10!white,draw=black}, 
				pbox/.style={inner sep=0pt, outer sep=0pt},
				block/.style={draw,rectangle, minimum width={0.8 cm},minimum height={1.2 cm},outer sep=2pt},
				pair/.style={draw,rectangle, minimum width={0.8 cm},minimum height={0.8 cm},outer sep=-3pt,fill=gray!10!white},
				paraedgeRB/.style={-Latex,double=red!30!white,draw=blue,double distance=2pt},
				paraedgeR/.style={-Latex,double=red!30!white,draw=black,double distance=2pt},
				paraedgeB/.style={-Latex,double,draw=blue,double distance=2pt},
				paraedge/.style={-Latex,double,double distance=2pt},
			}
			\begin{tikzpicture}[scale=2]
				\draw[every node/.style={draw,circle}]
				(00:1) node(123){123}
				(60:1) node(132){132}
				(120:1) node(213){213}
				(180:1) node(231){231}
				(240:1) node(312){312}
				(300:1) node(321){321}
				;
				\draw[-latex](123)--(231);
				\draw[-latex](123)--(132);
				\draw[-latex](123) to[out=-30,in=30,looseness=3](123);
				
				\draw[-latex](213)to[bend left=10](231);
				\draw[-latex](213)to[bend left=10](132);
				\draw[-latex](213)--(123);
				
				\draw[-latex](312)to[bend left=10](231);
				\draw[-latex](312) to[bend left=10] (132);
				\draw[-latex](312)--(123);
				
				\draw[-latex](132)to[bend left=10](213);
				\draw[-latex](132) to[bend left=10] (312);
				\draw[-latex](132)--(321);
				
				\draw[-latex](231)to[bend left=10](213);
				\draw[-latex](231)to[bend left=10](312);
				\draw[-latex](231)--(321);
				
				\draw[-latex](321)--(213);
				\draw[-latex](321)--(312);
				\draw[-latex](321)to[out=-30,in=30,looseness=3] (321);

				\begin{scope}[scale=0.2,xshift=14cm,yshift=2cm]
					\tikzstyle{every node}=[inner sep=0pt, outer sep=0pt]
					\draw (0,-2.5) node[pair](121){}; 
					\draw (0,-1)  node(123){123};
					\draw (0,-2)  node(231){231};
					\draw (0,-3)  node(132){132};
					\draw (0,-2) node[block](12){};
					\draw (-2,-2) node{\it 12};

					\begin{scope}[xshift=6cm,yshift=0cm]
						\draw (0,-1.5) node[pair](212){}; 
						\draw (0,-1)  node(213){213};
						\draw (0,-2)  node(312){312};
						\draw (0,-3)  node(321){321};
						\draw (0,-2) node[block](21){};
						\draw (2,-2) node{\it 21};
					\end{scope}
					
					\draw[-latex] (213) to[out=180,in=20] (12);
					\draw[-latex] (312) to[out=180,in=10] (12);

					\draw[-latex,transform canvas={yshift=0ex}](231) to[out=-10,in=190] (21);
					\draw[-latex,transform canvas={yshift=0ex}](132) to[out=-10,in=200] (21);

					\draw[-latex](321) to[out=0,in=270,looseness=5] (21);
					
					\draw[-latex](123) to[out=180,in=90,looseness=5] (12);
				\end{scope}
				
			\end{tikzpicture}
		\end{center}
		\caption{The graph of overlapping $3$-permutations on the left and the cluster graph for $3$-permutations on the right.}
		\label{fig:overlapgraph}
	\end{figure}
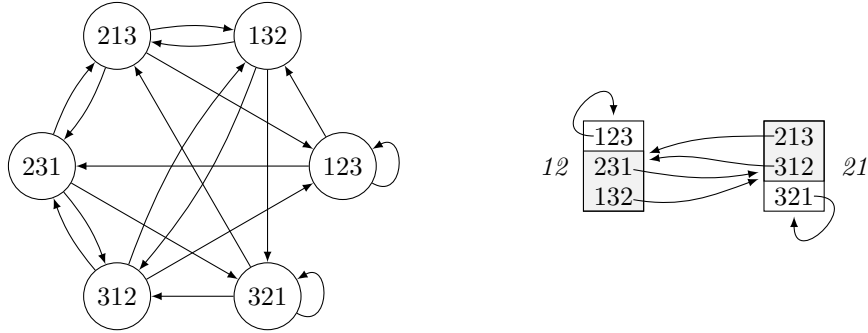

	The cluster graph, where a \universal{}cycle corresponds to an Eulerian tour rather than a Hamiltonian cycle, has also been used to study universal cycles for $S_n$; see \cite{CDG92}. The \emph{cluster graph} for $n$-permutations, denoted here by $G$, identifies all permutations $x$ and $y$ where $\reduce(x_1\cdots x_{n-1})=\reduce(y_1\cdots y_{n-1})$. This means the edges are grouped by their origin. Its vertices are the clusters of $n$-permutations whose first $n-1$ entries are order-isomorphic, and its edges are the $n$-permutations: it has an edge from $\sigma$ to $\tau$ for each $n$-permutation $x$ satisfying $\sigma=\reduce(x_1\cdots x_{n-1})$ and $\tau = \reduce(x_2\cdots x_n)$. This means that each cluster contains $n$ permutations, and the cluster graph is a directed multigraph. For example, on the right in Figure~\ref{fig:overlapgraph} is the cluster graph for $3$-permutations, where each edge corresponds to the $3$-permutation at its head. Compressing the parallel edges of the cluster graph for $n$-permutations to single edges yields the graph of overlapping $(n-1)$-permutations. For example, the cluster graph for $4$-permutations is shown in Figure \ref{fig:clusters}, and after compressing parallel edges it is the same as the graph of overlapping $3$-permutations shown on the left of Figure \ref{fig:overlapgraph}. The cluster graph $G$ is balanced and strongly connected, as observed in \cite{CDG92}.
	
	Any Eulerian tour in the cluster graph gives a Hamiltonian cycle in the graph of overlapping $n$-permutations. It is conjectured that these Hamiltonian cycles can be extended to \universal{}cycles for $n$-permutations; see~\cite{CDG92,H90}.

	For words, translating a Hamiltonian cycle in the De Bruijn graph to a De Bruijn sequence is straightforward since the \universal{}cycle uses just $n$ letters.
	This is not the case for $n$-permutations. A \universal{}cycle or a \universal{}word for $n$ permutations may use many more distinct entries than $n$. Building a \universal{}cycle for $n$-permutations by following a Hamiltonian cycle in the graph of overlapping $n$-permutations could possibly lead to a situation where the beginning and the end are not compatible. 
	To illustrate this potential misalignment, consider the following simple example.
	Recall $S_3$ has a \universal{}word  
	14524314 that can be turned into \universal{}cycle by removing the last two letters. On the other hand 
	14625415 is still a \universal{}word, but removing the last two letters does not turn it into a \universal{}cycle since it would contain a sub-word 414.
	It is, however, easy to construct a \universal{}word by following a Hamiltonian path.
	
	The proof of Theorem~\ref{thm:uword} utilizes the cluster graph of $n$-permutations and performs a compression on the cluster graph. 
	We use the same shortening ideas as Kitaev, Potapov, and Vajnovszki~\cite{KPV19}.  
	In order to obtain a \universal{}cycle instead of a \universal{}word,
	we identify a particular part of the cluster graph that allows for the beginning and the end of the word coming from traversing an Eulerian tour to be connected.
	We use many of the lemmas as Kitaev, Potapov, and Vajnovszki~\cite{KPV19}, we include also proofs for completeness.

	\section{Proof of Theorem~\ref{thm:conj8}}
	
	We start by introducing and restating some technical definitions. Then we describe several properties of the cluster graph from Kitaev, Potapov, and Vajnovszki~\cite{KPV19}. Finally, we describe the procedure to obtain a shortened \universal{}cycle.

	Two permutations $\pi_1\cdots \pi_n$ and $\sigma_1\cdots \sigma_n$ are called \emph{twins} if they belong to the same cluster and $|\pi_1-\pi_n| = |\sigma_1-\sigma_n| = 1$. 
	For example, $3142$ and $2143$ are twins, and $134562$ and $234561$ are twins. 
	Pairs of twins are depicted in Figure~\ref{fig:clusters} in gray boxes, and they correspond to a pair of parallel edges. 
	
	The cluster graph $G$ is balanced and strongly connected, as observed in \cite{CDG92}. Recall $G$ has an edge $e$ from $\sigma$ to $\tau$ for each $n$-permutation $x$ satisfying $\sigma=\reduce(x_1\cdots x_{n-1})$ and $\tau = \reduce(x_2\cdots x_n)$. We write $L(e)$ for $x$. Later, when we compress pairs of twins and have only one edge $e$ from $\sigma$ to $\tau$, we write $L(e)$ for $x=x_1\cdots x_n \in \mathbb{N}^n$, where $x_1=x_n$, $\sigma=\reduce(x_1\cdots x_{n-1})$, and $\tau = \reduce(x_2\cdots x_n)$.

	\tikzset{
		similarp/.style={fill=gray!20!white,draw=black}, 
		pbox/.style={inner sep=0pt, outer sep=0pt},
		block/.style={draw,rectangle, minimum width={0.8 cm},minimum height={1.6 cm},outer sep=2pt},
		pair/.style={draw,rectangle, minimum width={0.8 cm},minimum height={0.8 cm},outer sep=-3pt,fill=gray!10!white},
		paraedgeRB/.style={-Latex,double=red!30!white,draw=blue,double distance=2pt},
		paraedgeR/.style={-Latex,double=red!30!white,draw=black,double distance=2pt},
		paraedgeB/.style={-Latex,double,draw=blue,double distance=2pt},
		paraedge/.style={-Latex,double,double distance=2pt},
		paraedgeX/.style={-Latex,transform canvas={yshift=+0.3ex}},
		paraedgeY/.style={-Latex,transform canvas={yshift=-0.3ex}},
	}
	\newcommand{\bshift}{0.5}

	\tikzset{edge_color8/.style={color=red,line width=1.2pt,postaction={draw,blue,line width=1.2pt,decorate,decoration={simple line,raise=1.2pt},postaction={draw,green,line width=1.2pt,decorate,decoration={simple line,raise=-2.4pt}}}}}
	
	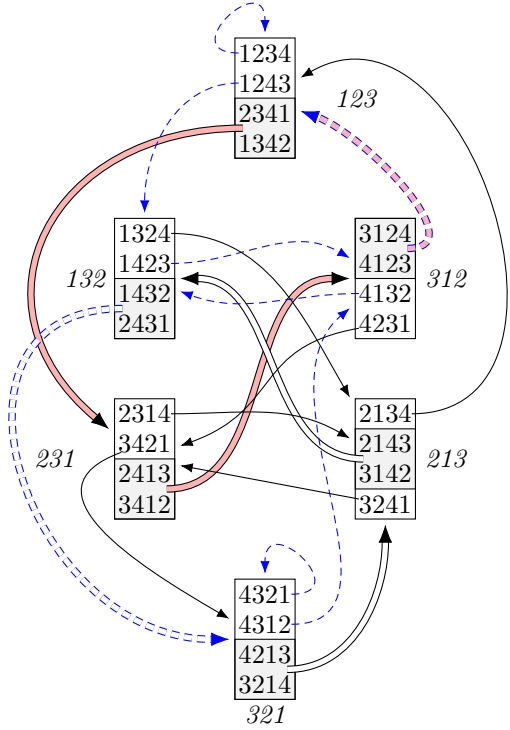
\begin{figure}
		\begin{center}
			\begin{tikzpicture}[scale=0.4]
				\clip (-9,-5) rectangle (9,20);
				\tikzstyle{every node}=[inner sep=0pt, outer sep=0pt]
				\draw (0,-2.5) node[pair](3213){}; 
				\draw (0,0)  node(4321){4321};
				\draw (0,-1)  node(4312){4312};
				\draw (0,-2)  node(4213){4213};
				\draw (0,-3)  node(3214){3214};
				\draw (0,-1.5) node[block](321){};
				\draw (0,-4) node{\it 321};
				
				\begin{scope}[xshift=4cm,yshift=6cm]
					\draw (0,-1.5) node[pair](2132){}; 
					\draw (0,0)  node(2134){2134};
					\draw (0,-1)  node(2143){2143};
					\draw (0,-2)  node(3142){3142};
					\draw (0,-3)  node(3241){3241};
					\draw (0,-1.5) node[block](213){};
					\draw (2,-1.5) node{\it 213};
				\end{scope}

				\begin{scope}[xshift=-4cm,yshift=6cm]
					\draw (0,-2.5) node[pair](2312){}; 
					
					\draw (0,0)  node(2314){2314};
					\draw (0,-1)  node(3421){3421};
					\draw (0,-2)  node(2413){2413};
					\draw (0,-3)  node(3412){3412};
					\draw (0,-1.5) node[block](231){};
					\draw (-3,-1.5) node{\it 231};
				\end{scope}
				
				\begin{scope}[xshift=-4cm,yshift=12cm]
					\draw (0,-2.5) node[pair](1321){}; 
					
					\draw (0,0)  node(1324){1324};
					\draw (0,-1)  node(1423){1423};
					\draw (0,-2)  node(1432){1432};
					\draw (0,-3)  node(2431){2431};
					\draw (0,-1.5) node[block](132){};
					\draw (-2,-1.5) node{\it 132};
				\end{scope}
				
				\begin{scope}[xshift=4cm,yshift=12cm]
					\draw (0,-0.5) node[pair](3123){}; 
					\draw (0,0)  node(3124){3124};
					\draw (0,-1)  node(4123){4123};
					\draw (0,-2)  node(4132){4132};
					\draw (0,-3)  node(4231){4231};
					\draw (0,-1.5) node[block](312){};
					\draw (2,-1.5) node{\it 312};
				\end{scope}

				\begin{scope}[xshift=0cm,yshift=18cm]
					\draw (0,-2.5) node[pair](1231){}; 
					\draw (0,0)  node(1234){1234};
					\draw (0,-1)  node(1243){1243};
					\draw (0,-2)  node(2341){2341};
					\draw (0,-3)  node(1342){1342};
					\draw (0,-1.5) node[block](123){};
					\draw (3,-1.5) node{\it 123};
				\end{scope}

				\draw[paraedgeRB,densely dashed](3123) to[out=0,in=-20] (123);
				\draw[paraedgeR](1231) to[out=180,in=140,looseness=1.5](231);
				\draw[paraedgeR](2312) to[out=0,in=180] (312);
				
				\draw[paraedge](3213) to[out=0,in=270] (213);
				\draw[paraedge](2132) to[out=180,in=0] (132);
				\draw[paraedgeB,densely dashed](1321) to[out=180,in=180,looseness=1.5] (321);

				\draw[-Latex,draw=blue,densely dashed](1234) to[out=180,in=90,looseness=5] (123);
				\draw[-Latex,draw=blue,densely  dashed](1243) to[out=180,in=90,looseness=1] (132);
				\draw[-Latex,draw=blue,densely  dashed](1423) to[out=0,in=150,looseness=1] (312);
				\draw[-Latex,draw=blue,densely  dashed](4132) to[out=180,in=-20,looseness=1] (132);
				\draw[-Latex,draw=blue,densely  dashed](4321) to[out=0,in=90,looseness=5] (321);
				\draw[-Latex,draw=blue,densely  dashed](4312) to[out=0,in=220,looseness=1] (312);

				\draw[-Latex](3241) to (231);
				\draw[-Latex](3421) to[out=200,in=150,looseness=1.4] (321);
				\draw[-Latex](2314) to[out=0,in=150] (213);
				\draw[-Latex](2134) to[out=0,in=30,looseness=1.4] (123);
				\draw[-Latex](1324) to[out=0,in=120,looseness=1] (213);
				\draw[-Latex](4231) to[out=190,in=20,looseness=1.3] (231);
				
			\end{tikzpicture}
		\end{center}
		\caption{Cluster graph for $4$-permutations. 
			Grey boxes indicate pairs of twins. 
			A double edge represents parallel edges corresponding to the permutations in the gray box.
			Double edges with filled insides and white insides form two cycles that could be used for shortening the \universal{}cycle.
			Blue dashed edges denote permutations in the family $\mathcal{P}^\star$, defined in the upcoming equation (\ref{eq:pstar}). They are used in the gluing part depicted in Figure~\ref{fig:glueZ}.
		}
		\label{fig:clusters}
	\end{figure}

	The proof of Theorem~\ref{thm:conj8} is constructive. Here is a sketch of the proof, omitting compression of twins.
	\begin{sketch}~
		
		\begin{enumerate}
			\item In the cluster graph $G$, we identify a cycle given by permutations $\mathcal{P}$ defined in the upcoming Eq \eqref{eq:G}.
			\item We remove them from $G$ and find a word $w$ whose $n$-windows are the permutations of $S_n\setminus \mathcal{P}$, exactly once each.
			\begin{enumerate}
				\item Lemma \ref{lem:connected} guarantees that $G-\mathcal{P}$ is strongly connected, which we use to show $G-\mathcal{P}$ is Eulerian.
				\item Given an Eulerian trail of $G-\mathcal{P}$, Lemma~\ref{lem:trail} shows how to build $w$. 
			\end{enumerate}
			\item Finally, we extend $w$ to a cyclic word that includes $\mathcal{P}$ in Lemma~\ref{lem:glue}.
		\end{enumerate}
	\end{sketch}
	
	A series of Lemmas in \cite{KPV19} describe the following additional properties of the cluster graph.
	
	\begin{lemma}[Kitaev, Potapov, and Vajnovszki~\cite{KPV19}]\label{lem:kpv}
		The following are true for each cluster graph $G$.
		\begin{enumerate}
			\item[(i)] Each cluster has exactly one pair of twins.
			\item[(ii)] For each cluster $X$, there exists a unique cluster $Y$ such that there are two edges from $X$ to $Y$. Also, there are no clusters $X$ and $Y$ such that there are three or more edges from $X$ to $Y$.
			\item[(iii)] For each cluster $Y$, there exists a unique cluster $X$ such that there are two edges from $X$ to $Y$.
			\item[(iv)] Any of the disjoint cycles formed by the double edges goes through exactly $n-1$ distinct clusters.
		\end{enumerate}
	\end{lemma}

	Let $\mathcal{P}$ be the following set of $2n$ permutations:
	\begin{align}\label{eq:G}
		\mathcal{P} = 
		\left\{  
		\begin{aligned}
			&~(n, n-1, \dots, 2, 1),       &  &~(1, 2, \dots, n - 2, n, n-1), \\
			&~(n, n-1, \dots, 3, 1, 2),    &  &~(1, 2, \dots, n - 3, n, n-2, n-1), \\
			&~(n, n-1, \dots, 4, 1, 2, 3), &  &~(1, 2, \dots, n-4, n, n-3, n-1, n-2),\\
			&~\hskip 5em \vdots                    &  &~\hskip 5em \vdots  \\       
			&~(n,n-1,\dots,k,1,\dots,k-1),  & &~(1,2,\dots,k,n,k+1,n-1,n-2\dots,k+2),\\
			&~\hskip 5em \vdots                     &  &~\hskip 5em \vdots  \\    
			&~(n, n-1,1, 2, \dots, n-2),   &  &~(1,n, 2, n-1, n-2, \dots, 3),  \\
			&~(n, 1, 2, \dots, n-1),       &  &~(n, 1, n-1, n-2, \dots, 2),  \\
			&~(1, 2, \dots, n),            &  &~(1, n, n-1, n-2, \dots, 2)    \\
		\end{aligned}
		\right\}.
	\end{align}
	The permutations in $\mathcal{P}$ form a tour in the cluster graph when read first from top to bottom of the first column, and then from top to bottom of the second column. For example, when $n=6$ we have
	\begin{align}
		\mathcal{P} = 
		\left\{  
		\begin{aligned}
			&~(6, 5, 4, 3, 2, 1), &   &~(1, 2, 3, 4, 6 ,5), \\
			&~(6, 5, 4, 3, 1, 2), &   &~(1, 2, 3, 6, 4, 5), \\
			&~(6, 5, 4, 1, 2, 3), &   &~(1, 2, 6, 3, 5, 4),\\
			&~(6, 5, 1, 2, 3, 4), &   &~(1, 6, 2, 5, 4, 3),  \\
			&~(6, 1, 2, 3, 4, 5), &   &~(6, 1, 5, 4, 3, 2),  \\
			&~(1, 2, 3, 4, 5, 6), &   &~(1, 6, 5, 4, 3, 2)    \\
		\end{aligned}
		\right\}.
	\end{align}
	We will use  $\mathcal{P}$ to glue together the ends of a \universal{}word for $S_n \setminus \mathcal{P}$, creating a \universal{}cycle for $S_n$. To accommodate shortening, we also need to consider twins of permutations in  $\mathcal{P}$. The only permutations in $\mathcal{P}$ that have twins are
	$(n, 1, 2, \dots, n-1)$ and $(1, n, n-1, n-2, \dots, 2)$.
	Let
	\begin{align}\label{eq:pstar}
		\mathcal{P}^\star = \mathcal{P} \cup \{(n-1,1,2,3,\ldots,n-2,n),   (2,n,n-1,\ldots,4,3,1)\}.
	\end{align}
	Hence $\mathcal{P}^\star$ is obtained from $\mathcal{P}$ by adding twin permutations to those already in $\mathcal{P}$.
	Blue dashed edges in Figure~\ref{fig:clusters} correspond to $\mathcal{P}^\star$.

	\begin{remark}\label{cl:213}
		By inspection of cases, for $n \geq 4$, the permutations in $\mathcal{P}^\star$ do not contain a $3$-window that is order-isomorphic to $213$ or $231$.
	\end{remark}

	If $TW$ is a set of twins in a cluster graph $G$, the \emph{compressed cluster graph} for $TW$ is obtained from $G$ by replacing each pair of parallel edges in $TW$,
	labeled by $L$ as $(x_1x_2\cdots x_{n-1}x_{n})$ and $(x_nx_2\cdots x_{n-1}x_{1})$,
	by one edge $e$ with $L(e) =\reduce(x_1x_2\cdots x_{n-1}x_{1})$. See Figure~\ref{fig:shortening} for an illustration where $TW$ forms a cycle.
	
	\begin{figure}
		\begin{center}
			\begin{tikzpicture}[scale=0.4]
				\tikzstyle{every node}=[inner sep=0pt, outer sep=0pt]

				\begin{scope}[xshift=-4cm,yshift=6cm]
					\draw (0,-2.5) node[pair](2312){}; 
					
					\draw (0,0)  node(2314){2314};
					\draw (0,-1)  node(3421){3421};
					\draw (0,-2)  node(2413){2413};
					\draw (0,-3)  node(3412){3412};
					\draw (0,-1.5) node[block](231){};
					\draw (-2,-1.5) node{\it 231};
				\end{scope}

				\begin{scope}[xshift=4cm,yshift=6cm]
					\draw (0,-0.5) node[pair](3123){}; 
					\draw (0,0)  node(3124){3124};
					\draw (0,-1)  node(4123){4123};
					\draw (0,-2)  node(4132){4132};
					\draw (0,-3)  node(4231){4231};
					\draw (0,-1.5) node[block](312){};
					\draw (2,-1.5) node{\it 312};
				\end{scope}
				
				\begin{scope}[xshift=0cm,yshift=12cm]
					\draw (0,-2.5) node[pair](1231){}; 
					\draw (0,0)  node(1234){1234};
					\draw (0,-1)  node(1243){1243};
					\draw (0,-2)  node(2341){2341};
					\draw (0,-3)  node(1342){1342};
					\draw (0,-1.5) node[block](123){};
					\draw (3,-1.5) node{\it 123};
				\end{scope}

				\draw[-latex](3124) to[out=0,in=-30] (123);
				\draw[-latex](4123) to[out=0,in=-10,looseness=1.4] (123);
				\draw[-latex](2341) to[out=180,in=140,looseness=1.5](231);
				\draw[-latex](1342) to[out=180,in=130,looseness=1.5](231);
				\draw[-latex](2413) to[out=0,in=175] (312);
				\draw[-latex](3412) to[out=0,in=190] (312);

				\begin{scope}[xshift=1cm]
					
					\begin{scope}[xshift=10cm,yshift=6cm]
						\draw (0,-2.5) node[pair](2312){2312}; 
						\draw (0,0)  node(2314){2314};
						\draw (0,-1)  node(3421){3421};
						\draw (0,-1.5) node[block](231){};
						\draw (-2,-1.5) node{\it 231};
					\end{scope}

					\begin{scope}[xshift=18cm,yshift=6cm]
						\draw (0,-0.5) node[pair](3123){3123}; 
						\draw (0,-2)  node(4132){4132};
						\draw (0,-3)  node(4231){4231};
						\draw (0,-1.5) node[block](312){};
						\draw (2,-1.5) node{\it 312};
					\end{scope}
					
					\begin{scope}[xshift=12cm,yshift=12cm]
						\draw (0,-2.5) node[pair](1231){1231}; 
						\draw (0,0)  node(1234){1234};
						\draw (0,-1)  node(1243){1243};
						\draw (0,-1.5) node[block](123){};
						\draw (3,-1.5) node{\it 123};
					\end{scope}

					\draw[-latex](3123) to[out=20,in=-10,looseness=1.5] (123);
					\draw[-latex](1231) to[out=180,in=140,looseness=1.5](231);
					\draw[-latex](2312) to[out=0,in=175] (312);
				\end{scope}
			\end{tikzpicture}
		\end{center}
		\caption{The effect of compressing one cycle in the cluster graph for $n = 4$. See     Figure~\ref{fig:clusters} for the entire cluster graph.}
		\label{fig:shortening}
	\end{figure}
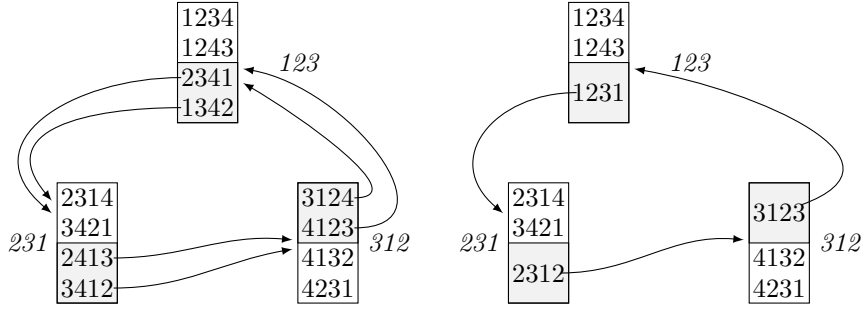
	
	\begin{lemma}\label{lem:connected} 
		For $n \geq 4$, the graph obtained from a compressed cluster graph by removing all edges corresponding to $\mathcal{P}^\star$ is strongly connected.
	\end{lemma}
	\begin{proof}
		Notice that a compressed cluster graph is obtained from the (uncompressed) cluster graph $G$ by replacing parallel edges with single edges. Hence it is enough to show that $G-\mathcal{P}^\star$ is strongly connected. We prove there exists a walk between any two vertices in $G-\mathcal{P}^\star$.
		
		Let $a=a_1a_2\cdots a_{n-1}$ and $b=b_1b_2\cdots b_{n-1}$ be permutations in $S_{n-1}$ corresponding to two vertices in $G-\mathcal{P}^\star$. 
		We construct a word $Q$ that corresponds to a walk in $G-\mathcal{P}^\star$ starting at $a$ and ending at $b$. 
		We define $Q=axb'$, where $x$ is an integer and $b'$ is created from $b$ as follows:
		\[
		x = \begin{cases} 
			n  & \text{ if } a_{n-2} > a_{n-1},\\
			0  & \text{ if } a_{n-2} < a_{n-1},
		\end{cases}
		\hskip 4em
		b_i' = \begin{cases} 
			b_i + n  & \text{ if } b_i > (b_1+b_2)/2,\\
			b_i - n & \text{ otherwise.}
		\end{cases}
		\]
		See Figure~\ref{fig:transition} for sketch of the construction of a general $Q$ and for a particular example see Figure~\ref{fig:transitionsexample}.
		Notice that $\reduce(b') = b$ and $\{\reduce(a_{n-2}a_{n-1}x), \reduce(xb_{1}b_2)\} \subseteq \{231,213\}$. 
		Hence any $n$-permutation in $Q$ contains $x$ and contains a consecutive triple that reduces to $213$ or $231$. 
		Neither of these are contained in any permutation in $\mathcal{P}^\star$ by Claim~\ref{cl:213}. 
		Therefore, $Q$ corresponds to a path in $G-\mathcal{P}^\star$ from $a$ to $b$, and hence $G-\mathcal{P}^\star$ is strongly connected. 
	\end{proof}
	
	\tikzset{
		pvtx/.style={inner sep=1.7pt, outer sep=0pt, circle, fill=black,draw=black}, 
	}
	\begin{figure}
		\begin{center}
			\begin{tikzpicture}[scale=0.5]
				\draw 
				(0,0) node[pvtx,label=left:$a_{n-1}$]{}
				(-1,1) node[pvtx,label=left:$a_{n-2}$]{}
				(-3,-\bshift) rectangle (\bshift,2+\bshift)
				(1,3) node[pvtx,fill=blue,label=below:$x$]{}
				(2,5) node[pvtx,label=above:$b_1$]{}
				(3,-2) node[pvtx,label=below:$b_2$]{}
				(2-\bshift,-2+\bshift) rectangle (5+\bshift,-4-\bshift)
				(2-\bshift,4-\bshift) rectangle (5+\bshift,6+\bshift)
				(-2,-1) node{$a$}
				(4,3) node{$b'$}
				(4,-1) node{$b'$}
				
				(-5,6) node{$2n-1$}
				(-5,5) node{$\vdots$}
				(-5,4) node{$n+1$}
				(-5,3) node{$n$}
				(-5,2) node{$n-1$}
				(-5,1) node{$\vdots$}
				(-5,0) node{1}
				(-5,-1) node{$0$}
				(-5,-2) node{$-1$}
				(-5,-3) node{$\vdots$}
				(-5,-4) node{$-n+1$}
				
				(1,-6) node{(i)}
				;
				
				\begin{scope}[xshift=12cm]
					\draw
					(0,1) node[pvtx,label=left:$a_{n-1}$]{}
					(-1,0) node[pvtx,label=left:$a_{n-2}$]{}
					(-3,-\bshift) rectangle (\bshift,2+\bshift)
					(1,-1) node[pvtx,fill=blue,label=below:$x$]{}
					(2,-2) node[pvtx,label=below:$b_1$]{}
					(3, 5) node[pvtx,label=below:$b_2$]{}
					(2-\bshift,-2+\bshift) rectangle (5+\bshift,-4-\bshift)
					(2-\bshift,4-\bshift) rectangle (5+\bshift,6+\bshift)
					(-2,-1) node{$a$}
					(4,3) node{$b'$}
					(4,-1) node{$b'$}
					
					(1,-6) node{(ii)}
					;
				\end{scope}
			\end{tikzpicture}
		\end{center}
		\caption{Transitioning between two clusters in Lemma~\ref{lem:connected}.
			(i) depicts $a_{n-1} < a_{n-2}$ and $b_1 > b_2$
			and (ii) depicts $a_{n-1} > a_{n-2}$ and $b_1 < b_2$.
			The remaining two cases are similar.}
		\label{fig:transition}
	\end{figure}
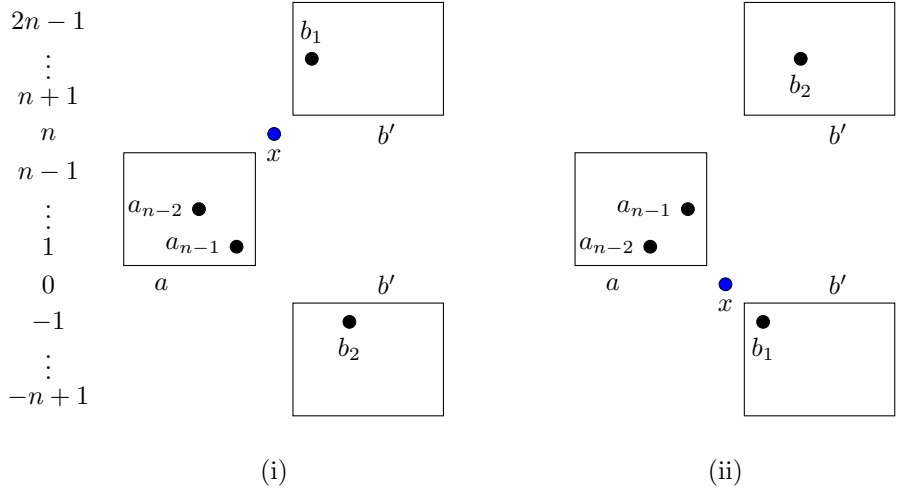

	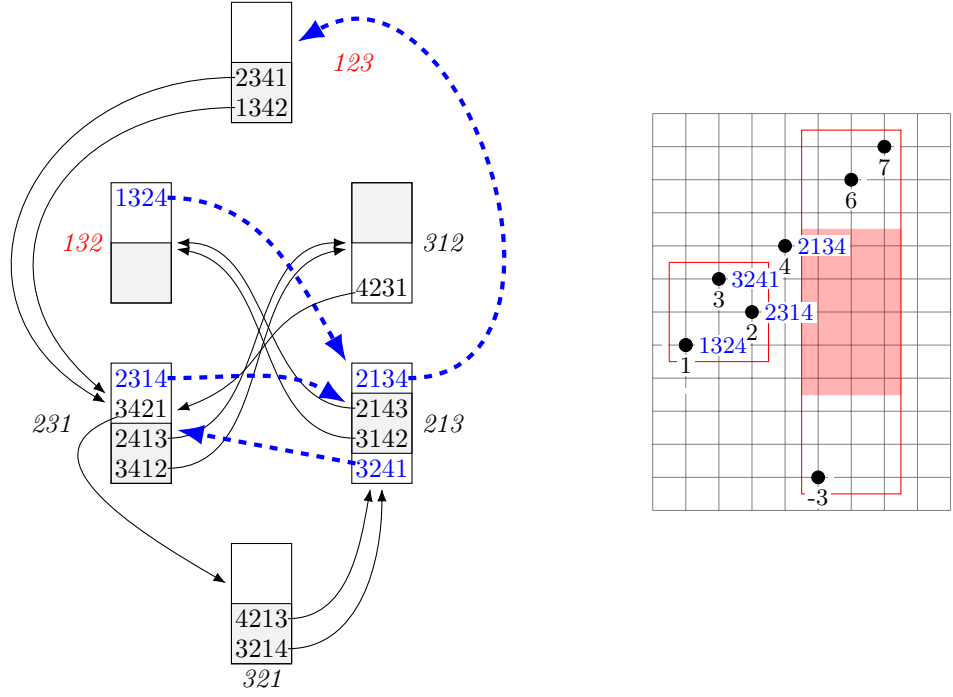
\begin{figure}[h]
		\begin{center}
			\begin{tikzpicture}[scale=0.4]
				\tikzstyle{every node}=[inner sep=0pt, outer sep=0pt]
				\draw (0,-2.5) node[pair](3213){}; 
				\draw (0,-2)  node(4213){4213};
				\draw (0,-3)  node(3214){3214};
				\draw (0,-1.5) node[block](321){};
				\draw (0,-4) node{\it 321};
				
				\begin{scope}[xshift=4cm,yshift=6cm]
					\draw (0,-1.5) node[pair](2132){}; 
					\draw[blue] (0,0)  node(2134){2134};
					\draw (0,-1)  node(2143){2143};
					\draw (0,-2)  node(3142){3142};
					\draw[blue] (0,-3)  node(3241){3241};
					\draw (0,-1.5) node[block](213){};
					\draw (2,-1.5) node{\it 213};
				\end{scope}

				\begin{scope}[xshift=-4cm,yshift=6cm]
					\draw (0,-2.5) node[pair](2312){}; 
					
					\draw[blue] (0,0)  node(2314){2314};
					\draw (0,-1)  node(3421){3421};
					\draw (0,-2)  node(2413){2413};
					\draw (0,-3)  node(3412){3412};
					\draw (0,-1.5) node[block](231){};
					\draw (-3,-1.5) node{\it 231};
				\end{scope}
				
				\begin{scope}[xshift=-4cm,yshift=12cm]
					\draw (0,-2.5) node[pair](1321){}; 
					
					\draw[blue] (0,0)  node(1324){1324};
					\draw (0,-1.5) node[block](132){};
					\draw[red] (-2,-1.5) node{\it 132};
				\end{scope}
				
				\begin{scope}[xshift=4cm,yshift=12cm]
					\draw (0,-0.5) node[pair](3123){}; 
					\draw (0,-3)  node(4231){4231};
					\draw (0,-1.5) node[block](312){};
					\draw (2,-1.5) node{\it 312};
				\end{scope}

				\begin{scope}[xshift=0cm,yshift=18cm]
					\draw (0,-2.5) node[pair](1231){}; 
					\draw (0,-2)  node(2341){2341};
					\draw (0,-3)  node(1342){1342};
					\draw (0,-1.5) node[block](123){};
					\draw[red] (3,-1.5) node{\it 123};
				\end{scope}

				\draw[-Latex](1342) to[out=180,in=140,looseness=1.5](231);
				\draw[-Latex](2341) to[out=180,in=150,looseness=1.5](231);
				
				\draw[-Latex](2413) to[out=0,in=180] (312);
				\draw[-Latex](3412) to[out=0,in=190] (312);
				
				\draw[-Latex](3214) to[out=0,in=270] (213);
				\draw[-Latex](4213) to[out=0,in=260] (213);
				
				\draw[-Latex](3142) to[out=180,in=-10] (132);
				\draw[-Latex](2143) to[out=180,in=0] (132);

				\draw[-Latex,blue,dashed, ultra thick](3241) to (231);
				\draw[-Latex](3421) to[out=200,in=150,looseness=1.4] (321);
				\draw[-Latex,blue,dashed,ultra thick](2314) to[out=0,in=150] (213);
				\draw[-Latex,blue,dashed,ultra thick](2134) to[out=0,in=30,looseness=1.4] (123);
				\draw[-Latex,blue,dashed,ultra thick](1324) to[out=0,in=120,looseness=1] (213);
				\draw[-Latex](4231) to[out=190,in=20,looseness=1.3] (231);

				\begin{scope}[xshift=13cm, yshift=6cm]
					\begin{scope}[scale=1.1]
						\draw[help lines](0,-4) grid (9,8);
						\draw[red] (0.5, 0.5) rectangle(3.5,3.5);
						\draw[red] (4.5,-3.5) rectangle(7.5,7.5);
						\draw[red,fill=red,opacity=0.3] (4.5,-0.5) rectangle(7.5,4.5);
						\foreach \x/\y/\z in {1/1/1324,2/3/3241,3/2/2314,4/4/2134,5/-3/,6/6/,7/7/ ,}{
							\draw(\x,\y) node[pvtx,label=right:{ }]{};
							\draw(\x,\y-0.3) node[below,fill=white,inner sep=1pt]{\small \y};
							\draw(\x+0.3,\y) node[right,fill=white,inner sep=1pt]{\color{blue}\small \z};
						}
					\end{scope}
				\end{scope}
			\end{tikzpicture}
		\end{center}
		\caption{Cluster graph for $4$-permutations with $\mathcal{P}^\star$ deleted. This figure demonstrates using the procedure in Lemma~\ref{lem:connected} to obtain a walk from $a=132$ to $b=123$, depicted using dashed edges. Here $n=4$, $x=4$, $b' = (-3,6,7)$, and $Q = (1,3,2,4,-3,6,7)$. Reducing the $4$-windows of $Q$ yields the walk $(1324, 3241, 2314, 2134)$. On the right is a picture of $Q$ with red rectangles bounding regions for $a$ and $b'$. The filled red rectangle indicates a region not containing any letter from $b'$ regardless of the letters in $b$.}
		\label{fig:transitionsexample}
	\end{figure}

	For a word $w$, an \emph{$n$-window} of $w$ is a substring of $n$ consecutive letters of $w$. A \emph{trail} in a directed multigraph is a walk without repeated edges.
	
	\begin{lemma}\label{lem:trail}
		Let $T = (e_1,e_2,\ldots,e_\ell)$ be a trail in a compressed cluster graph for $n$-permutations.
		There exists a word $w=w_1\cdots w_{\ell+n-1}$, where
		$\reduce(w_i\cdots w_{i+n-1}) = L(e_i)$ for all $1 \leq i \leq \ell$.
	\end{lemma}
	\begin{proof}
		If $\ell=1$, we let $w=L(e_1)$.
		If $\ell \ge 2$, let $w'=w'_1\cdots w'_{\ell+n-2}$ be a word for $(e_1,e_2,\ldots,e_{\ell-1})$ by induction.
		Let $a_1a_2\cdots a_n$ be $L(e_\ell)$.
		If $a_n = a_1$, i.e. $e_\ell$ is a compressed edge, then 
		$w$ is obtained from $w'$ by appending $w'_{\ell}$, which is the letter corresponding to $a_1$.
		Otherwise, $a_1 \ne a_n$, and we may need to modify $w'$ and determine a letter to append as follows.
		
		If $a_n = n$, let $x = \max\{w'_i:1 \le i \le \ell+n-2\}+1$.
		If $a_n < n$, let $i$ be the index such that $a_i=a_n+1$ and
		$x = w'_{\ell-1+i}$ be the letter in $w'$ corresponding to $a_i$.
		Then we define $w=w_1\cdots w_{\ell+n-1}$ as
		\[
		w_i = \begin{cases}
			w'_i & \text{ if }  w'_i < x \text{ and } i \leq \ell+n-2, \\
			w'_i+1 & \text{ if }  w'_i \geq x \text{ and } i \leq \ell+n-2,  \\
			x      & \text{ if } i = \ell+n-1.
		\end{cases}
		\]
		Since all $n$-windows of  $w_1 \cdots w_{\ell+n-2}$ are order isomorphic to $n$-windows of $w'_1 \cdots w'_{\ell+n-2}$, $w$ still represents the same permutations as $w'$ in the part before the last letter. 
		The choice of $x$ makes the last $n$-window in $w$ be order isomorphic to $L(e_\ell)$. 
	\end{proof}
	
	The next lemma describes how to use $\mathcal{P}$ to turn a particular word $w$ into a cyclic word $z$. It adds $n+1$ letters at the beginning of $w$.
	The result is a cyclic word covering the permutations in $w$ and also all permutations in $\mathcal{P}$. 
	A result of the operation, where $w$ is indicated by a box, can be seen in the Appendix for $n=4$.

	\begin{lemma}\label{lem:glue}
		Let $n \geq 4$.
		Let $\mathcal{P}'$ be such that $\mathcal{P} \subseteq  \mathcal{P}' \subseteq \mathcal{P}^\star$. Let $w=w_1w_2\cdots w_k$ be a word with $k\geq n-1$ that covers each $n$-permutation in $S_n$ at most once, and let $\mathcal{W}$ be the set of $n$-permutations covered by $w$. If $\reduce(w_1w_2\cdots w_{n-1}) = \reduce(w_{k-n+2}\cdots w_{k-1}w_k) = (n-1,\ldots,2,1)$ and $\mathcal{W}\cap\mathcal{P}' = \emptyset$, 
		then there exists a cyclic word $z$ of length $k+n+1$ such that each permutation in $\mathcal{W} \cup \mathcal{P}'$ is covered by exactly one $n$-window of $z$. 
		The permutations in $\mathcal{P}'\setminus \mathcal{P}$ are covered in a compressed way.
	\end{lemma}
	\begin{proof}
		The cyclic word $z=z_1z_2\cdots z_{n+1+k}$ can be defined as follows; see Figure~\ref{fig:glueZ} for guidance:
		\begin{align*}
			z_i &= \min\{w_1,w_2,\ldots,w_k\} - n-1 + i \text{ for } 1 \leq i \leq n-1, \\
			z_n &= \max\{w_1,w_2,\ldots,w_k\} + 1, \\
			z_{n+1} &= \begin{cases}
				z_{n-1}+1  & \text{ if }  (2,n,n-1,\ldots,4,3,1) \notin \mathcal{P}',\\
				w_{n-1}    & \text{ otherwise,}
			\end{cases}\\
			z_{n+1+i} &= w_i \text{ for } 1 \leq i \leq k-1, \\
			z_{n+1+k} &= \begin{cases}
				w_k & \text{ if }  (n-1,1,2,3,\ldots,n-2,n) \notin  \mathcal{P}',\\
				z_{n-1} & \text{ otherwise.} 
			\end{cases}
		\end{align*}
		The $n$-windows of the cyclic word $z$ that start at indices $k+3$ through $n+1$ (the first $n-1$ of which wrap around to the beginning of $z$) cover each permutation in $\mathcal{P'}$ exactly once, in the order given by reading the left column in \eqref{eq:G} first and the right column second, and with the permutations in $\mathcal{P}'\setminus \mathcal{P}$ covered in a compressed way.
		Notice that the possible shortenings are obtained by $n$-windows starting at $z_{n+1}$ and $z_{n+1+k}$.
		
		Since $z_{n+2}\cdots z_{n+k} = w_1\cdots w_{k-1}$, all permutations in $\mathcal{W}$ except the last one containing $w_k$ are covered by $z$. In the case that $z_{n+1+k} = w_k$, the last one is there as well.
		We need to check that $\reduce(w_{k-n+1}\cdots w_k) = \reduce(z_{k+2}\cdots z_{n+1+k})$ when $z_{n+1+k} = z_{n-1}$. By hypothesis, $w_{k-n+2} > \cdots > w_{k-1} > w_k$ so, if $w_k \ge w_{k-n+1}$, then $\reduce(w_{k-n+1}\cdots w_k) = (1,n,n-1,\ldots,2) \in \mathcal{P}$ or $\reduce(w_{k-n+1}\cdots w_k) = (1,n-1,n-2,\ldots,1)$, in both cases contradicting $\mathcal{P} \cap \mathcal{W} = \emptyset$. Therefore, $w_k < w_{k-n+1}$, so $w_k < \min\{w_{k-n+1},\ldots,w_{k-1}\}$, and decreasing $w_k$ to $z_{n+1+k}=z_{n-1}=\min\{w_1,w_2,\ldots,w_k\}-2$ does not change the relative order of $w_{k-n+1}\cdots w_k$.
	\end{proof}

	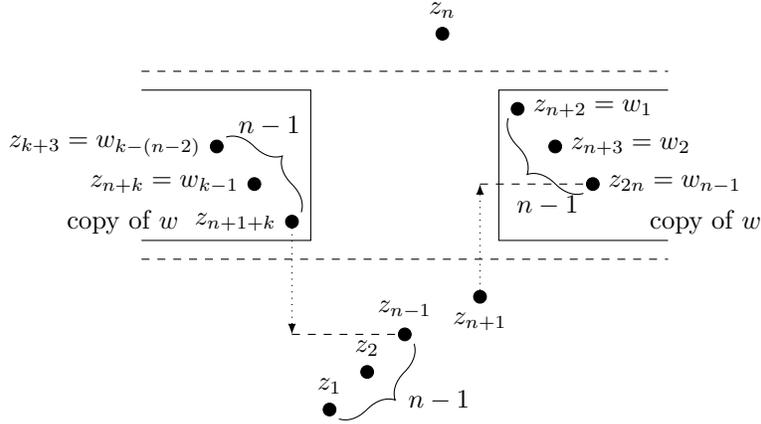
\begin{figure}[h]
		\begin{center}
			\begin{tikzpicture}[scale=0.5]
				\draw 
				(0,0) node[pvtx,label=left:$z_{n+1+k}$]{}
				(-1,1) node[pvtx,label=left:{$z_{n+k}=w_{k-1}$}]{}
				(-2,2) node[pvtx,label=left:{$z_{k+3}=w_{k-(n-2)}$}]{}
				(-4,-\bshift) -- (\bshift,-\bshift) -- (\bshift,3+\bshift)--(-4,3+\bshift)
				
				(10,-\bshift) -- (6-\bshift,-\bshift) -- (6-\bshift,3+\bshift) -- (10,3+\bshift)
				
				(-4.5,0) node{copy of $w$}
				(11,0) node{copy of $w$}
				
				(1,-5) node[pvtx,label=above:$z_1$]{}
				(2,-4) node[pvtx,label=above:$z_2$]{}
				(3,-3) node[pvtx,label=above:$z_{n-1}$]{}
				(4, 5) node[pvtx,label=above:$z_n$]{}
				(5,-2) node[pvtx,label=below:$z_{n+1}$]{}
				(6, 3) node[pvtx,label=right:{$z_{n+2}=w_1$}]{}
				(7, 2) node[pvtx,label=right:{$z_{n+3}=w_2$}]{}
				(8, 1) node[pvtx,label=right:{$z_{2n}=w_{n-1}$}]{}
				;
				
				\draw[dotted,-latex](0,0) -- (0,-3);
				\draw[dashed](0,-3) -- (3,-3);
				
				\draw[dotted,-latex](5,-2) -- (5,1);
				\draw[dashed](5,1) -- (8,1);
				
				\draw[dashed] (-4,-1) -- (10,-1);
				\draw[dashed] (-4,4) -- (10,4);

				\draw [decorate,decoration={brace,amplitude=10pt,mirror,raise=5pt},yshift=0pt]
				(1,-5) -- node [black,pos=0.5,xshift=0.3cm,label=below right:{$n-1$}] {}  (3,-3);
				
				\draw [decorate,decoration={brace,amplitude=10pt,raise=5pt},yshift=0pt]
				(-2,2) -- node [black,pos=0.5,yshift=0.4cm,xshift=0.2cm,label=above:{$n-1$}] {}  (0,0);
				
				\draw [decorate,decoration={brace,amplitude=10pt,mirror,raise=5pt},yshift=0pt]
				(6,3) -- node [black,pos=0.5,xshift=-0.1cm,yshift=-0.4cm,label=below:{$n-1$}] {}  (8,1);
				
			\end{tikzpicture}
		\end{center}
		\caption{The word $z$ from Lemma~\ref{lem:glue} for $n=4$ depicted.
			Dashed lines indicate relative order of entries.
			Arrows indicate the ``otherwise'' cases for $z_{n+1+k}$ and $z_{n+1}$.}
		\label{fig:glueZ}
	\end{figure}
	
	Now we are ready to prove Theorem~\ref{thm:conj8}, restated below.
	\MainThm*
	
	\begin{proof}[Proof of Theorem~\ref{thm:conj8}] 
		If $n=3$, then possible \universal{}cycles are for example 1232 and 145243 for $i=1$ and $i=0$, respectively.
		Let $n \geq 4$ and $0 \leq i \leq (n-2)!$.
		Let $G$ be the cluster graph for $n$-permutations.
		By Lemma~\ref{lem:kpv}, $G$ contains $(n-1)!/(n-1)=(n-2)!$ disjoint cycles formed by double edges.
		Let $C_i$ be a union of $i$ of these cycles.
		Let $\mathcal{P}'$ be the set satisfying $\mathcal{P}\subseteq\mathcal{P}'\subseteq \mathcal{P}^\star$, where each permutation in $\mathcal{P}^\star\setminus\mathcal{P}$ is included in $\mathcal{P}'$ if and only if it is a twin in $C_i$. 
		Let $G_i$ be the compressed cluster graph for the set of twins in $C_i$.
		Notice that $G_i$ has $n!-i(n-1)$ edges and is still Eulerian since cycles of double edges were replaced by cycles of single edges. 
		Let $T$ be a tour in $G_i$ corresponding to all of the permutations in $\mathcal{P}'$.
		Let $G_i'$ be obtained from $G_i$ by removing $T$.
		Then $G_i'$ is balanced because $G_i$ and $T$ are balanced, and $G_i'$ is strongly connected as Lemma~\ref{lem:connected} implies $G_i'$ has a strongly connected spanning subgraph, so $G_i'$ is Eulerian.
		Hence, using Lemma~\ref{lem:trail} with an Eulerian trail of $G_i'$ that starts and ends at $(n-1,n-2,\ldots,1)$, there exists a word $w=w_1w_2\cdots w_k$ whose $n$-windows are $L(e)$ for the edges $e$ of $G_i'$ in the order of the Eulerian trail, and which in particular has 
		$\reduce(w_1w_2\cdots w_{n-1}) = \reduce(w_{k-n+2}\cdots w_k) = (n-1,n-2,\ldots,1)$.
		By Lemma~\ref{lem:glue}, there exists a cyclic word $z$, which covers each $n$-permutation exactly once, and covers the twins in $C_i$ in a compressed way. This cyclic word $z$ is a shortened universal cycle for $S_n$ of length $n!-i(n-1)$.
	\end{proof}

	\section{Conclusion}
	In this note we proved a conjecture of Kitaev, Potapov, and Vajnovszki~\cite{KPV19} on shortening \universal{}words by adding repeated elements. Our construction does not control how many different entries are in the resulting \universal{}cycle. It would be interesting to investigate the smallest number of symbols needed to create a shortened \universal{}cycle of a given length. As suggested in \cite{GKSZ19}, it would also be interesting to determine a greedy algorithm for constructing shortened universal cycles for $S_n$.
	
	The case $i=0$ in Theorem \ref{thm:conj8} gives a new way to construct universal cycles for $S_n$. It is an open question whether every Eulerian tour of the cluster graph for $n$-permutations corresponds to a universal cycle for $S_n$; see \cite{CDG92,H90}. Our proof shows that the Eulerian tours containing the specific tour for $\mathcal{P}$ all correspond to universal cycles for $S_n$.
	
	\section*{Acknowledgments}
	Work on this project started during the Research Training Group (RTG) rotation at Iowa State University in the spring of 2021. Rachel Kirsch, Clare Sibley, and Elizabeth Sprangel were supported by NSF grant DMS-1839918.
	The authors thank Dylan Fillmore, Bennet Goeckner, Kirin Martin, and Daniel McGinnis, for helpful discussions during early stages of this project. We also thank the reviewers for many valuable suggestions, which improved the presentation of the result.
	
	\bibliographystyle{abbrvurl}
	\bibliography{GucyclesBibliography}
	
	\section*{Appendix}
	\subsection{Constructions}
	Here we describe Figures~\ref{fig:u-cycle},
	\ref{fig:u-cycle-3}, and
	\ref{fig:u-cycle-6}
	depicting constructions coming from Theorem~\ref{thm:conj8} for $n=4$ of three different lengths.
	Recall that the construction of $w$ in the proof of Theorem~\ref{thm:conj8} uses an Eulerian tour but does not explicitly say which one. Hence there are many possible outputs of the described algorithm. Here we present one for each length of the resulting \universal{}cycle.
	
	In addition to writing a sequence of numbers forming a \universal{}cycle, we depict the word in a grid, where each node corresponds to a letter. The letter is depicted below as a number. The number with four digits to the right of each node is a permutation of 1234 whose representation starts at the given letter.
	Since the word is cyclic, the first three nodes are repeated at the end as gray nodes to make it easier to see the cycle.
	Blue permutations are in $\mathcal{P}^\star$.
	Red sequences are compressed permutations, where the first and  last entries are the same.
	The black rectangle denotes the word $w$ as in Figure~\ref{fig:glueZ}.
	
	We made no attempt to minimize the number of letters used in the depicted \universal{}cycles.

	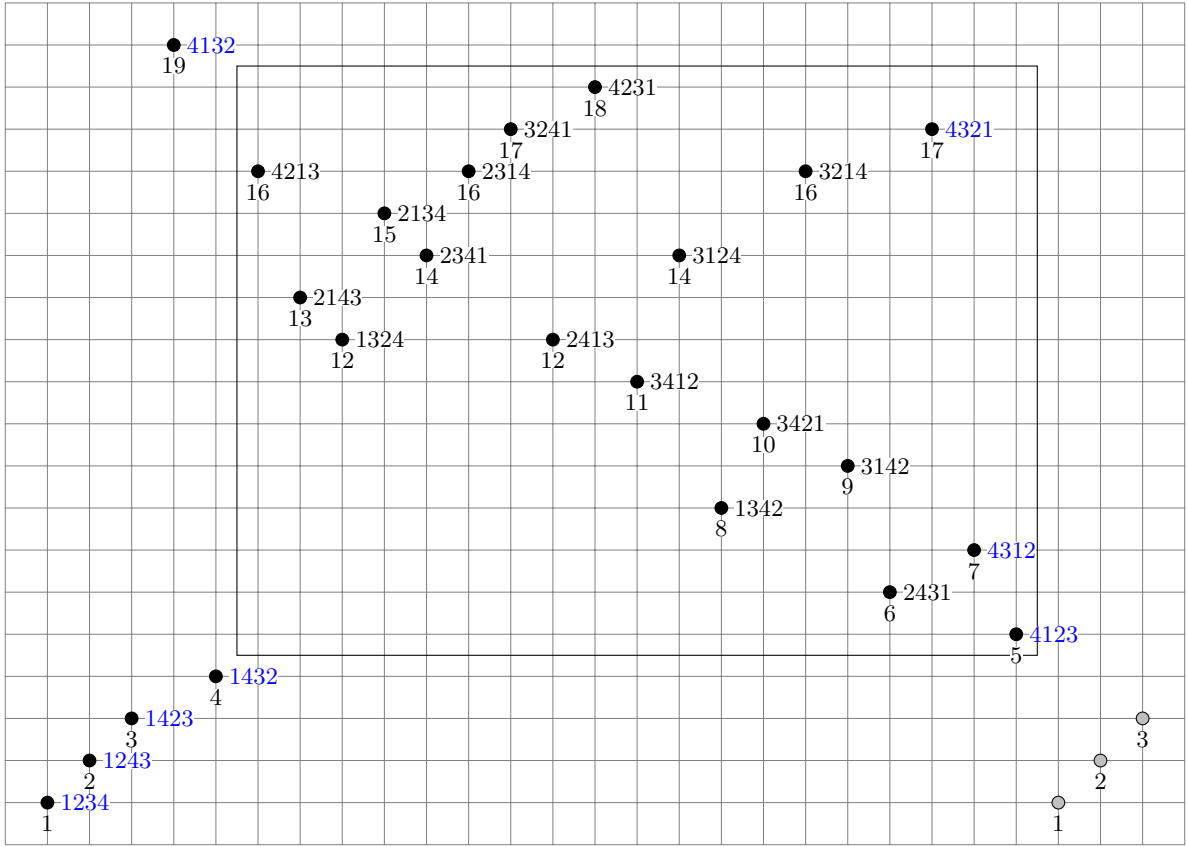
\begin{figure}
		\begin{center}
			\begin{tikzpicture}[scale=0.56]
				\draw[help lines](0,0) grid (28,20);
				\draw (5.5,4.5) rectangle (24.5,18.5);
				\foreach \x/\y/\z in {1/1/1234,2/2/1243,3/3/1423,4/19/4132,5/4/1432,22/17/4321,23/7/4312,24/5/4123}{
					\draw(\x,\y) node[pvtx,label=right:{ }]{};
					\draw(\x,\y-0.3) node[below,fill=white,inner sep=0]{\small \y};
					\draw(\x+0.3,\y) node[right,fill=white,inner sep=0]{\color{blue} \small \z};
				}
				\foreach \x/\y/\z in {6/16/4213,7/13/2143,8/12/1324,9/15/2134,10/14/2341,11/16/2314,12/17/3241,13/12/2413,14/18/4231, 15/11/3412,16/14/3124,17/8/1342,18/10/3421,19/16/3214,20/9/3142,21/6/2431}{
					\draw(\x,\y) node[pvtx]{};
					\draw(\x,\y-0.3) node[below,fill=white,inner sep=0]{\small \y};
					\draw(\x+0.3,\y) node[right,fill=white,inner sep=0]{\small \z};
				}
				\foreach \x/\y/\z in {25/1/1234,26/2/1243,27/3/1423}{
					\draw(\x,\y) node[pvtx,fill=gray!50!white,label=right:{ }]{};
					\draw(\x,\y-0.3) node[below,fill=white,inner sep=0]{\small \y};
				}
			\end{tikzpicture}
		\end{center}
		\caption{$(1, 2, 3, 19, 4, 16, 13, 12, 15, 14, 16, 17, 12, 18, 11, 14, 8, 10, 16, 9, 6, 17, 7, 5)$, a universal cycle for $S_4$.}
		\label{fig:u-cycle}
	\end{figure}
	
	\begin{figure}
		\begin{center}
			\begin{tikzpicture}[scale=0.56]
				\draw[help lines](0,0) grid (25,13);
				\draw (5.5,4.5) rectangle (21.5,11.5);
				\foreach \x/\y/\z in {1/1/1234,2/2/1243,3/3/1423,4/12/4132,5/4/1432,19/10/4321,20/8/4312}{
					\draw(\x,\y) node[pvtx,label=right:{ }]{};
					\draw(\x,\y-0.3) node[below,fill=white,inner sep=1pt]{\small \y};
					\draw(\x+0.3,\y) node[right,fill=white,inner sep=1pt]{\color{blue} \small \z};
				}
				\foreach \x/\y/\z in {6/10/3214,7/9/3214,8/8/2312,9/11/4231,10/7/2314,11/8/2134,12/6/1231,13/9/3421,14/10/4213, 15/6/2143,16/5/1324,17/9/3142,18/7/2431}{
					\draw(\x,\y) node[pvtx]{};
					\draw(\x,\y-0.3) node[below,fill=white,inner sep=1pt]{\small \y};
					\draw(\x+0.3,\y) node[right,fill=white,inner sep=1pt]{\small \z};
				}
				\foreach \x/\y/\z in {8/8/2312,12/6/1231,21/3/3123}{
					\draw(\x,\y) node[pvtx]{};
					\draw(\x,\y-0.3) node[below,fill=white,inner sep=1pt]{\small \y};
					\draw(\x+0.3,\y) node[right,fill=white,inner sep=1pt]{\color{red} \small \z};
				}
				\foreach \x/\y/\z in {22/1/1234,23/2/1243,24/3/1423}{
					\draw(\x,\y) node[pvtx,fill=gray!50!white,label=right:{ }]{};
					\draw(\x,\y-0.3) node[below,fill=white,inner sep=0]{\small \y};
				}
				\draw (0,-1) node[right]{};
			\end{tikzpicture}
		\end{center}
		\caption{$(1, 2, 3, 12, 4, 10, 9, 8, 11, 7, 8, 6, 9, 10, 6, 5, 9, 7, 10, 8, 3)$, a universal cycle for $S_4$ shortened by $3$.}
		\label{fig:u-cycle-3}
	\end{figure}
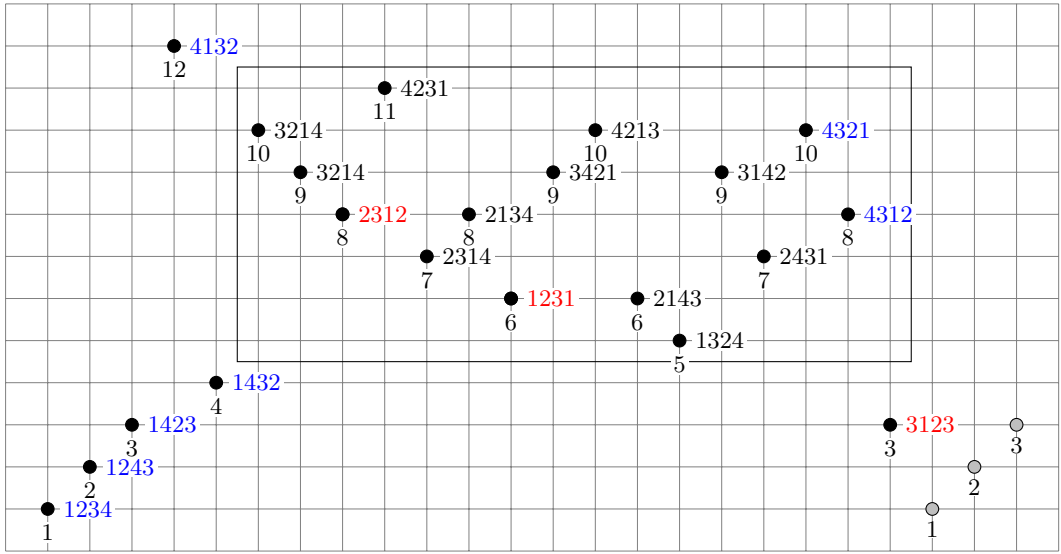

	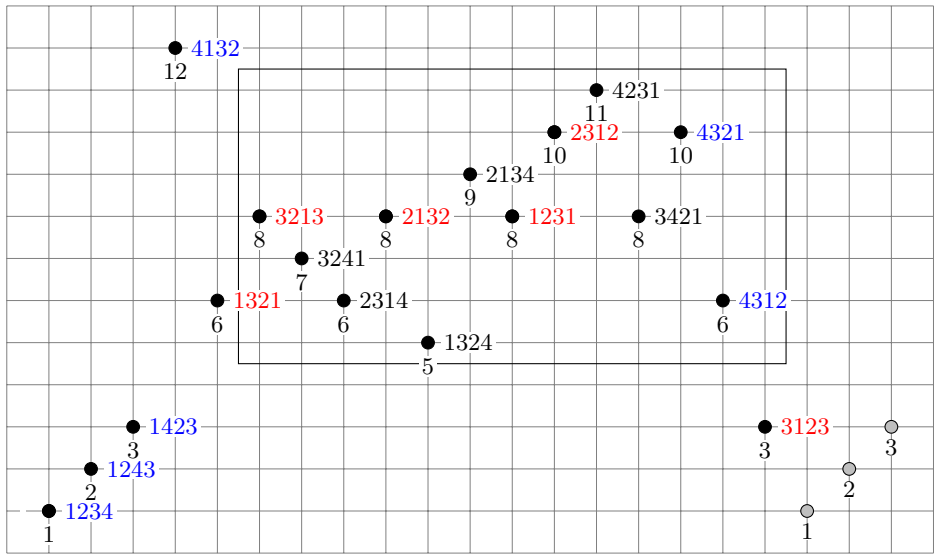
\begin{figure}
		\begin{center}
			\begin{tikzpicture}[scale=0.56]
				\draw[help lines](0,0) grid (22,13);
				\draw (5.5,4.5) rectangle (18.5,11.5);
				\foreach \x/\y/\z in {1/1/1234,2/2/1243,3/3/1423,4/12/4132,16/10/4321,17/6/4312,}{
					\draw(\x,\y) node[pvtx,label=right:{ }]{};
					\draw(\x,\y-0.3) node[below,fill=white,inner sep=1pt]{\small \y};
					\draw(\x+0.3,\y) node[right,fill=white,inner sep=1pt]{\color{blue} \small \z};
				}
				\foreach \x/\y/\z in {6/8/3213,7/7/3241,8/6/2314,9/8/2132,10/5/1324,11/9/2134,12/8/1231,13/10/2312,14/11/4231, 15/8/3421}{
					\draw(\x,\y) node[pvtx]{};
					\draw(\x,\y-0.3) node[below,fill=white,inner sep=1pt]{\small \y};
					\draw(\x+0.3,\y) node[right,fill=white,inner sep=1pt]{\small \z};
				}
				\foreach \x/\y/\z in {5/6/1321,6/8/3213,9/8/2132,12/8/1231,13/10/2312,18/3/3123}{
					\draw(\x,\y) node[pvtx]{};
					\draw(\x,\y-0.3) node[below,fill=white,inner sep=1pt]{\small \y};
					\draw(\x+0.3,\y) node[right,fill=white,inner sep=1pt]{\color{red} \small \z};
				}
				\foreach \x/\y/\z in {19/1/1234,20/2/1243,21/3/1423}{
					\draw(\x,\y) node[pvtx,fill=gray!50!white,label=right:{ }]{};
					\draw(\x,\y-0.3) node[below,fill=white,inner sep=0]{\small \y};
				}
				\draw (0,-1) node[right]{};
			\end{tikzpicture}
		\end{center}
		\caption{$(1, 2, 3, 12, 6, 8, 7, 6, 8, 5, 9, 8, 10, 11, 8, 10, 6, 3)$, a universal cycle for $S_4$ shortened by $6$.}
		\label{fig:u-cycle-6}
	\end{figure}

\end{document}